\documentclass[10pt]{amsart}

\usepackage{amsmath,amsthm,amsfonts, bbm, amssymb, hyperref,graphicx,color, accents}

\newcommand{\ii}{\mathbbm{i}}

\newcommand{\C}{\mathbb{C}}

\newcommand{\Sieg}{\mathcal{S}}

\newcommand{\Z}{\mathbb{Z}}
\newcommand{\Q}{\mathbb{Q}}

\newcommand{\norm}[1]{\left\vert #1 \right \vert}	
\newcommand{\Norm}[1]{\left\Vert #1 \right \Vert}

\renewcommand{\Re}{\operatorname{Re}}
\renewcommand{\Im}{\operatorname{Im}}

\newcommand{\Frac}{\mathbb K}

\def\[#1\]{\begin{align}#1\end{align}}
\def\(#1\){\begin{align*}#1\end{align*}}

\newtheorem{thm}{Theorem}
\numberwithin{thm}{section}

\newtheorem{lemma}[thm]{Lemma}

\theoremstyle{definition}

\theoremstyle{definition}

\numberwithin{equation}{section}

\allowdisplaybreaks

\title[Lagrange's Theorem]{Lagrange's Theorem for continued fractions on the Heisenberg group}

\author[J. Vandehey]{Joseph Vandehey}
\email{vandehey@uga.edu}

\begin{document}

\date{\today}

\maketitle

\begin{abstract}
We prove an analog of Lagrange's Theorem for continued fractions on the Heisenberg group: points with an eventually periodic continued fraction expansion are those that satisfy a particular type of quadratic form, and vice-versa.
\end{abstract}

\section{Introduction}

One of the strengths of the study of classical continued fractions is the connection between periodic expansions and quadratic irrationals. This came about in two parts. Euler's theorem states that any eventually periodic continued fraction expansion is a quadratic irrational, and Lagrange's theorem states that any quadratic irrational has an eventually periodic continued fraction expansion. It was a desire to extend these results to cubic and higher algebraic irrationals that inspired mathematicians to investigate multi-dimensional continued fractions; however, while analogs of Euler's theorem are somewhat easy, analogs of Lagrange's theorem are almost non-existent. To the best of the author's knowledge, the only multi-dimensional analog of Lagrange's theorem that the author relates to Klein polyhedra and sails \cite{Karpenkov}. (Weaker results than Lagrange's theorem, which, rather than characterizing eventually periodic continued fractions, show that algebraic irrationals of a particular form have eventually periodic continued fraction expansions, are far more common. See for example \cite{Garrityetal,Garrity}.) Schweiger is quite pessimistic on this topic, calling the question of classifying periodic continued fractions ``the most difficult problem in this area" \cite{Schweiger}.

In previous papers \cite{LV,Vandehey}, the author studied continued fractions on the Heisenberg group. We consider the Heisenberg group in its Siegel model, given by the space
\[\label{eq:planarsieg}
\Sieg := \{h=(u,v)\in \mathbb{C}^2: |u|^2-2\operatorname{Re} (v)=0\}
\]
with group law given by 
\(
(u_1,v_1)*(u_2,v_2)=(u_1+u_2,v_1+\overline{u_1}u_2+v_2) \qquad (u,v)^{-1} = (-u,\overline{v}).
\)
We denote the set of integer points in $\Sieg$ by $\Sieg(\Z)=\Sieg \cap \Z[\ii]^2$. We tend to denote integer points by $\gamma$ as opposed to $h$.  The Koranyi inversion $\iota$ is a conformal map on $\Sieg$ given by $\iota(u,v)=(-u/v,1/v)$ that corresponds to the inversion map $x \mapsto 1/x$ for real numbers.

Given an infinite sequence $\{\gamma_i\}_{i=0}^\infty$ of integer points, we define 
\(
\mathbb{K}\{\gamma_i\}_{i=0}^n = \gamma_0 \iota \gamma_1 \iota \dots \iota \gamma_n,
\)
suppressing $*$ and parentheses. We say that an infinite sequence $\{\gamma_i\}_{i=0}^\infty$ is a continued fraction expansion for a point $h\in \Sieg$ if 
\(
h = \lim_{n\to \infty} \mathbb{K}\{\gamma_i\}_{i=0}^n ,
\)
where the convergence of the limit here is in the Euclidean sense as points in $\mathbb{C}^2$.

Let $U(2,1;\Z[\ii])$ denote the space of matrices given by
\(
\{ M \in \operatorname{GL}_3(\mathbb{Z}) : M^\dagger J M = J\}
\)
where $\dagger$ denotes conjugate transpose and $J$ is the matrix
\(
J = \left( \begin{array}{ccc} 0 & 0 & -1 \\ 0 & 1 & 0 \\ -1 & 0 & 0 \end{array} \right) \in U(2,1;\Z[\ii]).
\)
We will see later how these matrices are linear fractional transformations on $\Sieg$. We say that a matrix $M \in U(2,1;\Z[\ii])$ is not a root of unity if $M^n \neq I$ for any positive integer $n$.

Given a point $h=(u,v)\in \Sieg$, we let $\vec{h}$ denote the \emph{vertical} vector whose transpose is $(1,u,v)$.

Our goal in this paper is to prove the following analog of Euler's and Lagrange's theorems on the Heisenberg group

\begin{thm}\label{thm:main}
Let $h\in \Sieg$. Then the following are equivalent. 
\begin{itemize}
\item There exists an eventually periodic sequence $\{ \gamma_i\}_{i=0}^\infty$ that is a continued fraction expansion for $h$.
\item There exists a matrix $M\in U(2,1;\Z[\ii])$ that is not a root of unity and satisfies
\[\label{eq:mainrelation}
\vec{h}^\dagger JM \vec{h} = 0.
\]
\end{itemize}
\end{thm}

In appearance, this looks rather different from the classical statements, especially comparing \eqref{eq:mainrelation} with the typical ``is a quadratic irrational." However, one way we think of a quadratic irrational $x$ is as a solution to a polynomial equation $Ax^2+Bx+C=0$, which may be rewritten in a form similar to \eqref{eq:mainrelation} as 
\(
(1,x)\left( \begin{array}{cc} 0 & -1 \\1 & 0 \end{array}\right)  \left( \begin{array}{cc} B/2 & A \\ -C & - B/2 \end{array}\right) \left( \begin{array}{c} 1 \\ x \end{array} \right) = 0.
\)
It is not immediately obvious, but such a relation can always be achieved with an integer matrix with determinant $\pm 1$: one could prove this using the same method that we do to show that $(1)\Rightarrow(2)$ in the proof of Theorem \ref{thm:main}.
 
We will note two weaknesses in Theorem \ref{thm:main} compared to the classical results.

First, the theorem only states the existence of some eventually periodic sequence $\{\gamma_i\}_{i=0}^\infty$ that is a continued fraction expansion for $h$. Typically we would like this to be the expansion that derives from some continued fraction algorithm, such as the algorithm with respect to the Dirichlet domain (see Section \ref{sec:algorithms}).

Second, it is not clear for which points $h\in \Sieg$ there will exist a matrix $M$ that satisfies \eqref{eq:mainrelation}. This is in contrast to the classical case where we know that the solution to irreducible quadratic polynomials are quadratic irrationals and vice-versa. We will show in Lemma \ref{lem:algebraic} that points $h$ which satisfy \eqref{eq:mainrelation} must have coordinates in an (at most) cubic extension of $\Q[\ii]$, but it is not clear whether any such point would satisfy \eqref{eq:mainrelation} for some matrix $M$.

\section{Background}

\subsection{More on the Heisenberg group}

The way we described the Heisenberg group in \eqref{eq:planarsieg} is known as \emph{the planar Siegel model}. We will also be interested in \emph{the projective Siegel model} given by
\(
\{(z_1:z_2:z_3) \in \C^3\setminus \{(0:0:0)\} : \norm{z_2}^2 - 2\Re{\overline{z_1}z_3} = 0\}
\)
with two points $(z_1:z_2:z_3)$ and $(z'_1:z'_2:z'_3)$ being considered the same if there is some non-zero constant $c\in \C$ such that $cz_i=z'_i$ for $i=1,2,3$.

We will freely switch back and forth between these two models, identifying a point $(u,v)$ in the planar Siegel model with the point $(1:u:v)$ in the projective Siegel model. We will also freely switch between points $(z_1:z_2:z_3)$ in the projective Siegel model and vertically written vectors \( \left( \begin{array}{c} z_1\\ z_2 \\ z_3 \end{array}\right), \) with context making it clear which of the two we mean.

One advantage of the projective model is that we can write rational points $(r/q,p/q)\in \Sieg \cap \Q[\ii]^2$ in the planar model as ``integer'' points $(q:r:p) \in \Z[\ii]^3$ in the projective model.

We note that the projective model includes a point at infinity $(0:0:1)$ which the planar version does not.

\subsection{More on continued fraction algorithms}\label{sec:algorithms}

The Heisenberg group comes equipped with a norm given by $\Norm{(u,v)}=\norm{v}^{1/2}$ and a distance given by $d(h,h') = \Norm{h^{-1}* h'}$. Note that this distance is topologically equivalent to Euclidean distance, so that convergence of a limit with respect to this distance is equivalent to convergence of a limit with respect to Euclidean distance.

We consider fundamental domains $K\subset \Sieg$ that satisfy the following conditions:
\begin{enumerate}
\item $K$ is a fundamental domain for $\Sieg$ under the action of left-translation by $\Sieg(\Z)$; i.e., $\bigcup_{\gamma \in \Sieg(\Z)} \gamma* K = \Sieg$, and $\gamma* K \cap K =\emptyset$ for all $\gamma \in \Sieg(\Z)\setminus \{(0,0)\}$.
\item $ \sup_{h\in K} \Norm{h} <1$.
\end{enumerate}

One such set is the Dirichlet domain $K_D$, defined as the set of points closer to $(0,0)$ than any other integer point, up to some choice of boundary.

Given a fundamental domain $K$ and a point $h$ we define $[h]$ as the point in $\Sieg(\Z)$ such that $[h]^{-1} * h\in K$. We thus think of $[h]$ as the nearest integer to $h$ with respect to $K$.

Thus, for a given fundamental domain $K$, we build analogs to classical continued fraction definitions in the following way. Let $T:K\to K$ be the Gauss map given by
\(
Th=\begin{cases}
(0,0), & \text{if }h=(0,0),\\
[\iota h]^{-1}*\iota h, & \text{otherwise.}
\end{cases}
\)
This closely resembles the classical Gauss map, which looks like $Tx = x^{-1}-\lfloor x^{-1}\rfloor$. Given a point $h\in \Sieg$ we define the forward iterates, $h_i\in K$, and the continued fraction digits, $\gamma_i\in \Sieg(\Z)$, for $h$ by 
\(
\gamma_0&=[h] & h_0 &= [h]^{-1}*h =\gamma_0^{-1}*h \\
\gamma_i &= [\iota h_{i-1}] & h_i&=T^i h_0= \gamma_{i}^{-1}*h_{i-1},
\)
with the sequence of continued fraction digits terminating if $h_i=(0,0)$. The Gauss map acts on the sequence of digits via a forward shift. The convergents are then given by
\(
\mathbb{K}\{\gamma_i\}_{i=0}^n  = \left( \frac{r_n}{q_n}, \frac{p_n}{q_n}\right) := \gamma_0*\left( \iota \gamma_1*\left( \iota\gamma_2*\dots \gamma_n\right)\right),
\)
where $r_n$, $p_n$, and $q_n$ are relatively prime Gaussian integers.

It was shown in \cite{LV} that if $h$ is an irrational point (i.e., not in $\Q[\ii]^2$), then it has an infinite number of continued fraction digits and that $\mathbb{K}\{\gamma_i\}_{i=0}^\infty$ exists and equals $h$. (In other words, the continued fraction digits generated in this way do, in fact, form a continued fraction expansion for the original point.)

\subsection{More on $U(2,1;\Z[\ii])$}

Given a point $h=(u,v)\in \Sieg$ and a matrix $M\in U(2,1;\Z[\ii])$, we will denote by $Mh$ the point whose projective coordinates are given by
\(
M\left( \begin{array}{c} 1\\u\\v \end{array}\right).
\)
In this sense, the matrices in $M$ can be thought of as linear fractional transformations on $\Sieg$.

There are several special kinds of matrices in $M$. The matrix \(
J = \left( \begin{array}{ccc}
0 & 0 & -1\\
0 & 1 & 0\\
-1 & 0 & 0
\end{array}  \right)
\) given in the definition of $U(2,1;\Z[\ii])$ can be seen to act on points in the same way as the Koranyi inversion $\iota$, so that $Jh = \iota h$.  We thus call $J$ the inversion matrix.

Similarly, given a point $h=(u,v)$, let $T_h$ denote the matrix
\(
\left( \begin{array}{ccc} 1 & 0 & 0 \\ u & 1 & 0 \\ v & \overline{u} & 1 \end{array}\right);
\)
then the point $T_h h'$ is equivalent to $h* h'$. This also implies that $T_h^{-1}=T_{h^{-1}}$ and $T_h T_{h'}=T_{h*h'}$. We thus call $T_h$ a translation matrix.

There are also the rotational matrices, which are the diagonal matrices in $U(2,1;\Z[\ii])$. We will often denote an arbitrary diagonal matrix by $D$. Because the determinants are $1$, this means all the elements on the diagonal are in the set $\{1,-1,\ii,-\ii\}$.

It can be shown that all matrices in $U(2,1;\Z[\ii])$ can be decomposed into a product of inversion, diagonal, and translation matrices. This is essentially one of the main results of \cite{FFLP}. We will use a similar method in Lemma \ref{lem:Mstructure}.

We will frequently label the coordinates of $M$ in the following way
\[\label{eq:Mform}
M=  \left( \begin{array}{ccc} Q'&\frak{Q}&-Q \\ R' & \frak{R} & -R \\P' & \frak{P} & -P \end{array}\right).
\]
and the coordinates of $M^i$ in the following way
\(
M^i =  \left( \begin{array}{ccc} {Q'}^{(i)}&\frak{Q}^{(i)}&-Q^{(i)} \\ {R'}^{(i)} & \frak{R}^{(i)} & -R^{(i)} \\ {P'}^{(i)} & \frak{P}^{(i)} & -P^{(i)} \end{array}\right).
\)

\section{Lemmas}

\subsection{More on the relation \eqref{eq:mainrelation}}

The relation $\vec{h}^\dagger JM \vec{h}=0$ is a bit mysterious on the surface, so we present a few lemmas which help to clarify it. In Lemma \ref{lem:relation}, we show that $\vec{v}_1^\dagger J \vec{v}_2=0$ if and only if $\vec{v}_1$ and $\vec{v}_2$ represent the same point in $\Sieg$ projectively. Using this, in Lemmas \ref{lem:iteratedrelation} and \ref{lem:inverse}, we show that $\vec{h}$ is an eigenvector for $M$, and thus $M$ can be replaced by $M^k$, $k\in \Z$, in \eqref{eq:mainrelation} without altering the truth of the statement. Lemma \ref{lem:nontranslation} implies that $M$ cannot be a translation matrix and satisfy this relation. (One can also quickly see that $M$ cannot be a diagonal matrix, as diagonal matrices are always roots of the identity.)

\begin{lemma}\label{lem:relation}
Suppose $(q:r:p),(q':r':p')\in \Sieg$. Then we have 
\[
\left(\begin{array}{c} q \\ r \\ p \end{array}\right)^\dagger J \left(\begin{array}{c} q' \\ r' \\ p' \end{array}\right) = 0
\]
 if and only if $(q:r:p)=(q':r':p')$ (projectively).
\end{lemma}

\begin{proof}
The ``if'' direction follows immediately, so we only need to prove the ``only if'' direction.

By writing out the relation, we have
\(
p'\overline{q} -r'\overline{r} + q'\overline{p}=0.
\)
Since both points are considered projectively, we may assume without loss of generality that $q=q'=1$. Thus we obtain
\[\label{eq:prpr}
p' -r'\overline{r} +\overline{p}=0.
\]
Since, both points are in $\Sieg$, we also have $|r|^2-2 \Re{p}=|r'|^2-2\Re{p'} =0$. So the real part of \eqref{eq:prpr} becomes
\(
0= \Re{p'}- \Re{r'\overline{r}} + \Re{\overline{p}} = \frac{1}{2} |r'|^2 - 2\Re{\overline{r'}r} + \frac{1}{2} |r'|^2 = \frac{1}{2}|r-r'|^2.
\)
Hence $r=r'$ and $-r'\overline{r} = -|r|^2$. 

The imaginary part of \eqref{eq:prpr} becomes
\(
0 = \Im{p'} + \Im{\overline{p}} = \Im{p'}-\Im{p}.
\)
Since the real parts of $p$ and $p'$ are already known to be equal (since $r=r'$), we therefore have that $p=p'$ and the lemma is proved.
\end{proof}

\begin{lemma}\label{lem:iteratedrelation}
Let $h\in \Sieg$ and $M\in  U(2,1;\Z[\ii])$ satisfy \eqref{eq:mainrelation}. Then for all positive integers $k$, $M^kh=h$ and 
\[\label{eq:iteratedrelation}
\vec{h}^\dagger J M^k\vec{h} = 0.
\]
\end{lemma}

\begin{proof}
First note that $M^k\vec{h}$ is the vertical vector with projective coordinates that is equivalent to the point $M^k h $ in planar coordinates. Also multiplying $\vec{h}$ by a non-zero constant does not alter the planar coordinates of the corresponding point.

By \eqref{eq:mainrelation} and Lemma \ref{lem:relation}, we know that $M\vec{h}$ must be the same point as $\vec{h}$ projectively. Thus, $\vec{h}$ is an eigenvector of $M$ and $M\vec{h}=\lambda\vec{h}$ for some  non-zero eigenvalue $\lambda\in \C$. But then, $M^k \vec{h}=\lambda^k \vec{h}$, and thus $M^k h = h$. Since $M^k \vec{h}$ represents the same point as $\vec{h}$ projectively, we have \eqref{eq:iteratedrelation} by Lemma \ref{lem:relation}.
\end{proof}

\begin{lemma}\label{lem:nontranslation}
If $h\in \Sieg$ is not the point at infinity and $\gamma\neq (0,0)$, then \eqref{eq:mainrelation} cannot hold for $M=T_\gamma$.
\end{lemma}

\begin{proof}
Suppose to the contrary that \eqref{eq:mainrelation} does hold for such a matrix. Then by Lemma \ref{lem:iteratedrelation}, we have that $T_\gamma h = h$, but this says that $\gamma* h = h$, which is only possible if $\gamma=(0,0)$ or if $h$ is the point at infinity, which is a contradiction.
\end{proof}

\begin{lemma}\label{lem:inverse}
If \eqref{eq:mainrelation} holds for a particular $h\in \Sieg$ and $M\in U(2,1;\Z[\ii])$, then it holds with $h$ and $M^{-1}$.
\end{lemma}

\begin{proof}
This follows by taking conjugate transposes of both sides of \eqref{eq:mainrelation} and by the fact that for $M\in U(2,1;\Z[\ii])$, we have $M^{\dagger}J = JM^{-1}$.
\end{proof}

\subsection{More on the structure of matrices $M\in U(2,1;\Z[\ii])$}

We require a few more lemmas to help us understand the matrices $M\in U(2,1;\Z[\ii])$. In addition to some basic facts about these matrices, we will pay special attention to what happens if one of the corner elements is a $0$ (Lemma \ref{lem:Mzeroes}) and to what happens when we take powers of $M$ (Lemma \ref{lem:nonzero}), and we will also show that $M^4$ can always be written as a product of translation and inversion matrices (Lemma \ref{lem:Mstructure}).

For the next two lemmas, the proof is purely computational and we omit it.

\begin{lemma}
We have for any $\gamma\in \Sieg$:
\(
(J T_{\gamma})^{\dagger} = J T_{\overline{\gamma}} ,
\)
where $\overline{(u,v)}=(\overline{u},v)$.
\end{lemma}

\begin{lemma}\label{lem:diagonal}
There are $8$ diagonal matrices in $ U(2,1;\Z[\ii])$, and they are
\(
 \left( \begin{array}{ccc} 1 & 0 & 0\\ 0 & \pm1 & 0\\ 0 & 0 & 1 \end{array} \right),  \left( \begin{array}{ccc} -1 & 0 & 0\\ 0 & \pm1 & 0\\ 0 & 0 & -1 \end{array} \right),  \left( \begin{array}{ccc} \ii & 0 & 0\\ 0 & \pm1 & 0\\ 0 & 0 & \ii \end{array} \right),  \left( \begin{array}{ccc} -\ii & 0 & 0\\ 0 & \pm1 & 0\\ 0 & 0 & -\ii \end{array} \right).
\)
If $D$ is a diagonal matrix, then $D J = J D$, and for any $\gamma\in \Sieg(\Z)$, there exists $\gamma' \in \Sieg(\Z)$ such that $D T_{\gamma} = T_{\gamma'} D$. In addition, for any diagonal matrix $D$, we have $D^4 = I$.
\end{lemma}

\begin{lemma}\label{lem:Mtransform}
Suppose $M\in U(2,1;\Z[\ii])$ takes the form \eqref{eq:Mform}. Then, we have
\(
\norm{R'}^2-2 \Re(\overline{P'}Q') &=0 & \norm{R}^2- 2 \Re(\overline{P}Q) &=0\\
-\overline{Q'}\frak{P}+\overline{R'}\frak{R}-\overline{P'}\frak{Q} &=0 & 
-\overline{Q'}P  +\overline{R'}R-\overline{P'}Q &= 1 
\)
so that $(Q':R':P')$ and $(Q:R:P)$ are in $\Sieg$. In addition, we have $
\norm{\frak{Q}}^2+2\Re(\overline{Q'}Q)=0  $
so that $(Q':\frak{Q}:-Q)\in \Sieg$ as well.
\end{lemma}

\begin{proof}
The first four equations come from writing out the elements of the matrix $M^\dagger J M$ and then comparing these to the corresponding elements of $J$. The final equation comes by the same method, but to the matrix $M^T$ instead, noting that $U(2,1;\Z[\ii])$ is preserved under transposition.
\end{proof}

\begin{lemma}\label{lem:Mzeroes}
Suppose $M$ is in $U(2,1;\Z[\ii])$ and that there exists some element in a corner of $M$ that equals $0$. Then the two elements adjacent to that element are also $0$, and the three elements on the diagonal adjacent those two must all have norm $1$.
\end{lemma}

\begin{proof}
We will prove the case where $Q'=0$. The other cases are proved by an identical method. 

If $Q'=0$, then by $\norm{R'}^2-2 \Re(\overline{P'}Q') =0$ and $
\norm{\frak{Q}}^2+2\Re(\overline{Q'}Q)=0  $ from  Lemma \ref{lem:Mtransform}, we have that both $R'$ and $\frak{Q}$ (the two adjacent elements) must also be $0$. Thus, $M$ has the form
\(
\left( \begin{array}{ccc}
0 & 0 & \omega_1\\
0 & \omega_2 & *\\
\omega_3 & * & * \end{array} \right).
\)
By taking determinants on both sides of $M^\dagger J M= J$ we see that $\norm{\det M}^2 = 1$. Since we also have that $\norm{\det{M}}=\norm{\omega_1\omega_2\omega_3}$ and that all elements of $M$ are in $\Z[\ii]$, we must have that $\norm{\omega_i}=1$ for $i=1,2,3$.
\end{proof}

\begin{lemma}\label{lem:nonzero}
Suppose that $M$ is not a root of the identity and satisfies \eqref{eq:mainrelation} for some $h \in \Sieg$. Then $\{M^i\}_{i=1}^\infty$ is a sequence of distinct matrices. The coordinates $Q^{(i)}$ and ${Q'}^{(i)}$ are $0$ for at most finitely many indices $i$. And at least one of $Q^{(i)}$ and ${Q'}^{(i)}$ tends to infinity in norm as $i$ tends to infinity.
\end{lemma}

\begin{proof}
Suppose $M^i$ and $M^j$ are the same matrix with $i<j$, then $M^{j-i}=I$, which is impossible since $M$ is not a root of the identity. Therefore $\{M^i\}_{i=1}^\infty$ is a sequence of distinct matrices.

Suppose ${Q}^{(i)}=0$ for some $i$, then Lemma \ref{lem:Mzeroes} implies that $M$ takes the form
\(
\left( \begin{array}{ccc} * & 0 & 0 \\ * &* & 0 \\ *&*&* \end{array} \right)
\)
In particular, we have $M^i =D T_{\gamma}$ for some diagonal matrix $D$ and some $\gamma\in \Sieg(\Z)$. But then by applying Lemma \ref{lem:diagonal} to $M^{4i}=DT_\gamma DT_\gamma DT_\gamma DT_\gamma$ to shift all the copies of $D$ to the end and remove them (since $D^4=I$), we have that $M^{4i}= T_{\gamma'}$ for some $\gamma' \in \Sieg(\Z)$. 
Lemma \ref{lem:iteratedrelation} implies that $M^{4i}=T_{\gamma'}$ satisfies \eqref{eq:mainrelation}, but Lemma \ref{lem:nontranslation} says this is impossible.

Suppose ${Q'}^{(i)}=0$ for some $i$, then again Lemma \ref{lem:Mzeroes} implies that $M^i = DT_{\gamma} J$ for some diagonal matrix $D$ and some $\gamma \in \Sieg(\Z)$. But this implies that $h= M^i h= D T_{\gamma}J h$. By rearranging, we obtain $D^{-1}h= T_{\gamma} Jh$. But for any diagonal matrix $D$, there is at most one $\gamma$ satisfying this relation. Since there are $8$ total distinct diagonal matrices, there are at most $8$ values $i$ for which ${Q'}^{(i)}=0$.

Now suppose  $\max_{i\in\mathbb{N}}\{\norm{{Q'}^{(i)}},\norm{{Q}^{(i)}}\}$ is finite. Since $|\frak{Q}^{(i)}|^2 -2 \Re(\overline{{Q}^{(i)}}{Q'}^{(i)})=0$, we have that there must exist distinct integers $i,j\in\mathbb{N}$ with $i<j$, such that $({Q'}^{(i)},{\frak{Q}}^{(i)},{Q}^{(i)})=({Q'}^{(j)},{\frak{Q}}^{(j)},{Q}^{(j)})$. But then $M^{(j-i)}=M^j (M^i)^{-1}=M^j J (M^i)^\dagger J$. Writing out the upper-right hand element (that is, $-Q^{(j-i)}$), we see that it is $|\frak{Q}^{(i)}|^2 -2 \Re(\overline{{Q}^{(i)}}{Q'}^{(i)})$, which equals $0$ by Lemma \ref{lem:Mtransform}. But this is a contradiction as we have already seen in this proof that $Q^{(j-i)}$ can never be $0$. Therefore $\max_{i\in\mathbb{N}}\{\norm{{Q'}^{(i)}},\norm{{Q}^{(i)}}\}$ cannot be finite.
\end{proof}

\begin{lemma}\label{lem:Mstructure}
If $M\in U(2,1;\Z[\ii])$, then there exists a sequence $\{\gamma_i\}_{i=0}^n$ with $\gamma_i\in \Sieg(\Z)$ for $0\le i \le n$ and $\gamma_i \neq (0,0)$ for $1\le i \le n-1$, such that $M^4=T_{\gamma_0}JT_{\gamma_1}J\dots JT_{\gamma_n}J$
\end{lemma}

\begin{proof}
Let $\{\gamma_i\}_{i=0}^{m}$ be the continued fraction digits of $(Q':R':P')$ with respect to the Dirichlet region. (It is possible that $\gamma_0=(0,0)$ if $(Q':R':P')\in K_D$.)  Let $M'$ be the matrix $T_{\gamma_0} J T_{\gamma_1} \dots J T_{\gamma_{m}}$. Since $M'(0,0)=(R'/Q',P'/Q')$ (see, for example, the proof of Lemma \ref{lem:evenperiodic}), the left-most column of $M'$ will be the vector $(Q',R',P')^T$, possibly up to multiplication by a unit.

Consider $(M')^{-1}M  = JM^{\dagger}JM$. The lower left corner of this matrix is $|R'|^2-2\Re{\overline{P'}Q'}$, which equals $0$ by Lemma \ref{lem:Mtransform}. Thus, by Lemma \ref{lem:Mzeroes}, $(M')^{-1} M$ can be written as $ J T_{\gamma_{m+1}} J D$, for some integer $\gamma_{m+1}\in \Sieg(\Z)$ which is possibly $(0,0)$ and some diagonal matrix $D$. Therefore, we have that $M$ can be written as $T_{\gamma_0} J T_{\gamma_1} \dots J T_{\gamma_{m+1}} J D$.

Applying the same trick we did in the proof of Lemma \ref{lem:nonzero}, we can remove the appearance of $D$ by raising $M$ to the fourth power and using Lemma \ref{lem:diagonal} to move all copies of $D$ to the end. Note that if $\gamma_0$ or $\gamma_{m+1}$ equals $(0,0)$, these can be removed from the middle of the expansion since $T_{\gamma}JT_{(0,0)} J T_{\gamma'} = T_{\gamma}T_{\gamma'} = T_{\gamma*\gamma'}$. If $\gamma*\gamma'=(0,0)$, the resulting matrix can be removed from the expansion in the same way.
\end{proof}

\subsection{Miscellaneous lemmas}

We need a few more results, which, not being easily categorized elsewhere, we collect here.

\begin{lemma}
Given two points $(u,v),(u',v')\in \Sieg$, we have 
\(
d((u,v),(u',v')) = \norm{v-u\overline{u'}+\overline{v'}}^{1/2}
\)
\end{lemma}

This is an immediate consequence of the definitions.

\begin{lemma}\label{lem:vrelation}
Suppose that there is a matrix $M$ of the form \eqref{eq:Mform} and a point $h=(u,v)\in\Sieg$ that satisfies \eqref{eq:mainrelation}. Then 
\[
 v\overline{Q}-u\overline{R}+\overline{P}=\frac{1}{  Q'+\frak{Q}u-Qv}.
\]
\end{lemma}

\begin{proof}
By Lemma \ref{lem:iteratedrelation}, we have that $M(u,v)=(u,v)$, and thus
\(
\left(   \frac{R'+\frak{R}u-Rv}{Q'+\frak{Q}u-Qv} , \frac{P'+\frak{P}u-Pv}{Q'+\frak{Q}u-Qv} \right)=(u,v)
\)
Thus, we have
\(
&v\overline{Q}-u\overline{R}  +\overline{P}\\ &= \frac{P'+\frak{P}u-Pv}{Q'+\frak{Q}u-Qv} \overline{Q}- \frac{R'+\frak{R}u-Rv}{Q'+\frak{Q}u-Qv}  \overline{R}  + \overline{P}\\
&= \frac{(\overline{Q}P'-\overline{R}R'+\overline{P}Q')+u(\overline{Q}\frak{P}-\overline{R}\frak{R}+\overline{P}\frak{Q})-v(\overline{Q}P-\norm{R}^2+\overline{P}Q)}{(Q'+\frak{Q}u-Qv)}\\
&= \frac{1}{(Q'+\frak{Q}u-Qv)},
\)
where the last equality follows from the relations in Lemma \ref{lem:Mtransform}
\end{proof}

\begin{lemma}\label{lem:qrelation}
Let $M$ be a matrix that satisfies $JT_{\gamma_1}JT_{\gamma_2}\dots JT_{\gamma_n}$. Let $(u_0,v_0), (u_n,v_n)$ be two points such that $(u_0,v_0)=M(u_n,v_n)$. Moreover, let $(u_i,v_i)$ be the point defined by $J T_{\gamma_{i+1}} J T_{\gamma_{i+2}} \dots J T_{\gamma_n} (u_n,v_n)$.

Then $ Q'+\frak{Q}u_n-Qv_n=(-1)^n \left( v_0 v_1 v_2 \dots v_{n-1} \right)^{-1} .$
\end{lemma}

The proof is identical to the proof of Lemma 3.19 in \cite{LV}.

\begin{lemma}\label{lem:evenperiodic}
Suppose that $h$ has an eventually periodic continued fraction expansion given by $\{\gamma_i\}_{i=0}^\infty$, and let $\gamma'_i\in \Sieg(\Z)$ for $0\le i \le j$ be a finite sequence of non-zero integer points (with $\gamma'_0$ possibly equal to $(0,0)$. Then 
\(
h'= T_{\gamma_0'} J T_{\gamma_1'} \dots J T_{\gamma_j'} h
\)
has an eventually periodic continued fraction expansion.
\end{lemma}

\begin{proof}
This would follow by induction if we could show that $Jh$ and $T_{\gamma}h$ (for $\gamma\in \Sieg(\Z)$) both have eventually periodic continued fraction expansions.

If $\gamma_0=(0,0)$, then $Jh$ has a continued fraction expansion given by $\{\gamma_1,\gamma_2,\gamma_3,\dots\}$ and $T_{\gamma}h$ has a continued fraction expansion given by $\{\gamma,\gamma_1,\gamma_2,\dots\}$.

If $\gamma_0\neq (0,0)$ then $Jh$ has a continued fraction expansion given by $\{(0,0),\gamma_0,\gamma_1,\dots\}$ and $T_{\gamma}h$ has a continued fraction expansion given by $\{\gamma*\gamma_0,\gamma_1,\gamma_2,\dots\}$.
\end{proof}

\begin{lemma}\label{lem:algebraic}
Suppose $h\in \Sieg$ satisfies \eqref{eq:mainrelation} for some $M\in U(2,1;\Z[\ii])$ that is not a root of the identity. Then there exists an algebraic irrational $\alpha$ of degree at most $3$ over $\Z[\ii]$ such that $h\in \Q[\ii][\alpha]$.
\end{lemma}

\begin{proof}
By Lemma \ref{lem:iteratedrelation}, we know that $h$ must be an eigenvector of $M$. By standard facts from linear algebra, the three eigenvalues of $M$, call them $\lambda_1,\lambda_2,\lambda_3$, must exist in $\Q[\ii][\alpha]$ for some algebraic $\alpha$ of degree at most $3$ over $\Z[\ii]$. Then, if we attempt to find an eigenvector 
\[
\left( \begin{array}{c} z_1\\ z_2\\z_3 \end{array}\right)
\]
using standard Gauss-Jordan elimination, the coordinates $z_1,z_2,z_3$ will all be in $\Q[\ii][\alpha]$.

Note that by Lemma \ref{lem:Mtransform}, the top and bottom rows of $M$ must be linearly independent, so Gauss-Jordan elimination will never result in a single linear equation of the type $Az_1+Bz_2+Cz_3=0$; however, it could result in two linear equations of the type $Az_1+Bz_2=0$ and $Cz_2+Dz_3=0$, in which case we would just let $z_1=1$ to achieve the desired result.

Thus, if $h$ corresponds to this eigenvalue, it will equal $(z_2/z_1,z_3/z_1)$, and by rationalizing the denominator, we obtain the desired result.
\end{proof}

%%%%%%%%%%%%%%%%%%%%%%
\section{Proof of Theorem \ref{thm:main}}
%%%%%%%%%%%%%%%%%%%%%%
$(1)\Rightarrow (2)$: Suppose there exists an eventually periodic sequence $\{\gamma_i\}_{i=0}^\infty$ with $\gamma_i \in \Sieg(\Z)$, such that $\Frac\{\gamma_i\}_{i=0}^\infty=h$. Suppose that the pre-periodic part of the sequence is $\{\gamma'_0, \gamma'_1, \dots, \gamma'_j\}$, and the periodic part is $\{\gamma''_1, \gamma''_2, \dots, \gamma''_k\}$. (We will always assume that $\gamma_0$ is in the pre-periodic part, shifting the period over by one position if necessary.) Let $M$ be the matrix given by 
\(
M=\left( T_{\gamma'_0} J T_{\gamma'_1} J T_{\gamma'_2} \dots J T_{\gamma'_j} \right) J T_{\gamma''_1} J T_{\gamma''_2} \dots J T_{\gamma''_k} \left( T_{\gamma'_0} J T_{\gamma'_1} J T_{\gamma'_2} \dots J T_{\gamma'_j} \right)^{-1}.
\)
The matrix $M$, applied to  $\Frac\{\gamma_i\}_{i=0}^\infty$ (which equals $(u,v)$), first removes the pre-periodic part, adds an extra copy of the periodic part, then reapplies the pre-periodic part, thereby returning us to the original point (again, see the proof of Lemma \ref{lem:evenperiodic}). Therefore, \eqref{eq:mainrelation} is true.

Moreover, suppose that there exists $m\ge 1$ such that $M^m = I$. Then we also have that
\[\label{eq:Bform}
\left(  J T_{\gamma''_1} J T_{\gamma''_2} \dots J T_{\gamma''_k} \right)^m = I.
\]
Applying both these matrices to $(0,0)$, we see that $(0,0)$ has a continued fraction expansion $\{\gamma''_{i\mod{k}}\}_{i=0}^{mk}$ (with $\gamma''_0=(0,0)$). Thus, all of the points $\Frac\{\gamma_i\}_{i=0}^{j+mkn}$ must be the same point in $\Sieg$ for $n\ge 0$. Since there is no element $\gamma\in \Sieg(\Z)$ such that $JT_{\gamma}=I$, we have by \eqref{eq:Bform} that $k$ and $m$ cannot both be $1$. Therefore,  $\Frac\{\gamma_i\}_{i=0}^{j+mkn+1}$ must all be the same point in $\Sieg$ for $n\ge 0$ and these must be distinct from  $\Frac\{\gamma_i\}_{i=0}^{j+mkn}$. Therefore $\Frac\{\gamma_i\}_{i=0}^{n}$ cannot converge as it is eventually periodic with at least two distinct points in its period. This is a contradiction.

$(2) \Rightarrow (1)$: \textbf{Step 1:} We will show first that it suffices to show that this holds when $M$ can be written as $J T_{\gamma_1} \dots J T_{\gamma_k}$ for some sequence of $\gamma_i\in \Sieg(\Z)$, $1\le i \le k$.

By Lemma \ref{lem:iteratedrelation}, we may replace $M$ in \eqref{eq:mainrelation} by $M^4$, and thus, by Lemma \ref{lem:Mstructure}, we may assume that $M$ decomposes in one of the following four ways (from here on we assume all $T_\gamma$ have $\gamma\neq (0,0)$):
\begin{enumerate}
\item $T_{\gamma_0} J T_{\gamma_1} \dots J T_{\gamma_N} J$.
\item $T_{\gamma_0} J T_{\gamma_1} \dots J T_{\gamma_N} $.
\item $ J T_{\gamma_1} \dots J T_{\gamma_N} J$.
\item $J T_{\gamma_1} \dots J T_{\gamma_N} $.
\end{enumerate} 

We next will show that $M$ can be written as $ABA^{-1}$ where $A=T_{\gamma_0'} J T_{\gamma_1'} \dots J T_{\gamma_j'}$ or $J T_{\gamma_1'} \dots J T_{\gamma_j'}$ and $B= J T_{\gamma_1''} \dots J T_{\gamma_k''}$ or as $T_{\gamma''}$. In essence, we are trying to show that we can decompose $M$ in the same way that it was composed in the proof of $(1)\Rightarrow (2)$.

\emph{Case 1} $M=T_{\gamma_0} J T_{\gamma_1} \dots J T_{\gamma_N} J$. 

In this case, we let $A=T_{\gamma_0}$ and $B=J T_{\gamma_1} \dots J T_{\gamma_N} J T_{\gamma_0}$.

\emph{Case 2} $M= T_{\gamma_0} J T_{\gamma_1} \dots J T_{\gamma_N}$. 

If $\gamma_0^{-1} \neq \gamma_N$, then we let $A=T_{\gamma_0}$ and $B=J T_{\gamma_1} \dots J T_{\gamma_N*\gamma_0^{-1}}$. Otherwise let $J$ be the largest integer strictly less than $N/2$ such that $\gamma_j^{-1} = \gamma_{N-j}$ for all $0\le j \le J$. Then we let $A=T_{\gamma_0} J T_{\gamma_1} J \dots J T_{\gamma_{J+1}}$ and $B= J T_{\gamma_{J+2}} J T_{\gamma_{J+2}} \dots J T_{\gamma_{N-J}*\gamma_{J+1}^{-1}}$. 

\emph{Case 3} $ M=J T_{\gamma_1} \dots J T_{\gamma_N} J$.

If $\gamma_1^{-1} \neq \gamma_N$, then we let $A=J T_{\gamma_1}$ and $B=J T_{\gamma_1} \dots J T_{\gamma_N*\gamma_1^{-1}}$. Otherwise, we let $J$ be the largest integer strictly less than $N/2$ such that $\gamma_j^{-1} = \gamma_{N+1-j}$ for all $1\le j \le J$. Then we let $A= J T_{\gamma_1} J \dots J T_{\gamma_{J+1}}$ and $B= J T_{\gamma_{J+2}} J T_{\gamma_{J+2}} \dots J T_{\gamma_{N+1-J}*\gamma_{J+1}^{-1}}$. 

\emph{Case 4} $M=J T_{\gamma_1} \dots J T_{\gamma_N} $. 

In this case we take $A=I$ and $B=M$.

Note that in cases 2 and 3, if $N$ is odd and $J=(N-1)/2$, then $B$ would be a translation matrix.

Now, we have
\(
\vec{h}^\dagger J M \vec{h}  &= \vec{h}^\dagger J ABA^{-1} \vec{h} \\
&=\left( A^{-1} \vec{h}\right)^\dagger J B \left( A^{-1} \vec{h} \right).
\)
Let $h'=A^{-1}h$. By Lemma \ref{lem:evenperiodic}, if $h'$ has an eventually periodic continued fraction expansion, so does $h$. Thus it suffices to show that $h'$ has an eventually periodic expansion.

By Lemma \ref{lem:nontranslation}, $B$ cannot be a translation matrix, and so must be written as $J T_{\gamma_1'} \dots J T_{\gamma_j'}$. Since no power of $M$ is the identity, no power of $B$ can be the identity either.

\textbf{Step 2:} We will now assume that $M$ can be written as $J T_{\gamma_1} \dots J T_{\gamma_k}$ for some sequence of $\gamma_i\in \Sieg(\Z)$, $1\le i \le k$.

 Note that by Lemma \ref{lem:relation}, we have that
\(
\vec{h}^\dagger J M^i \vec{h}=0
\)
for any positive integer $i$. Lemma \ref{lem:nonzero} applies and we may assume that we are always taking $i$ large enough so that ${Q'}^{(i)}$ and $-Q^{(i)}$ are non-zero.

By Lemma \ref{lem:qrelation}, we have
\[\label{eq:qik}
{Q'}^{(i)}+\frak{Q}^{(i)}u-Q^{(i)}v = (-1)^{ik} (v_0v_1v_2\dots v_{k-1})^{-i},
\]
utilizing the definition of $v_i$ from that lemma with $(u_0,v_0)=(u_i,v_i)=h$.

Suppose for a moment that $\norm{vv_1\dots v_{k-1}}=1$. We claim that $Q^{(i)}$ is not bounded in norm. If it were bounded, then ${Q'}^{(i)}$ would have to be unbounded in norm by Lemma \ref{lem:nonzero}. By Lemma \ref{lem:Mtransform}, we have that $\frak{Q}^{(i)}$ has norm at most $\sqrt{2\norm{\overline{Q^{(i)}}{Q'}^{(i)}}}$. Thus, the right-hand side of \eqref{eq:qik} will have norm $1$, and while the left-hand side will have unbounded norm, which is a contradiction and so $Q^{(i)}$ is unbounded in norm.

Case 1: $|v_0 v_1 \dots v_{k-1}| \ge 1$. Then we have that 
\(
d\left( (u,v), \left( \overline{\left( \frac{{\frak{Q}}^{(i)}}{Q^{(i)}}\right)}, -\overline{\left( \frac{{Q'}^{(i)} }{Q^{(i)} }\right)} \right)\right) = \norm{\frac{1}{(v_0 v_1 v_2 \dots v_{k-1})^i Q^{(i)}} }^{1/2}
\)
tends to $0$.   

Since the top row of $M$ is the vector $({Q'}^{(i)},\frak{Q}^{(i)}, -Q^{(i)})$, we have that the vector $ \left(\overline{Q^{(i)}},\overline{\frak{Q}^{(i)}},-\overline{{Q'}^{(i)}} \right)^T$ 
 is the left-hand column of $J(M^i)^{\dagger}$. This matrix expands as
\(
J(M^i)^{\dagger} &= J \left(  \left( J T_{\gamma_1} \dots J T_{\gamma_k} \right)^{\dagger}\right)^{i}\\
&= J \left( J T_{\overline{\gamma_k}} J T_{\overline{\gamma_{k-1}}} \dots J T_{\overline{\gamma_1}} \right)^{i}\\
&= \left( T_{\overline{\gamma_k}} J T_{\overline{\gamma_{k-1}}} \dots J T_{\overline{\gamma_1}} J \right)^{i} J.
\)
 Therefore, by applying this matrix to $(0,0)$, we see that $\left(\overline{Q^{(i)}}:\overline{\frak{Q}^{(i)}}:-\overline{{Q'}^{(i)}} \right)= \Frac\{\overline{\gamma_{k-j \mod{k}}}\}_{j=0}^{ik-1}$.

Let $\left(q_n:r_n:p_n\right)=\Frac\{\overline{\gamma_{k-j \mod{k}}}\}_{j=0}^n$. To prove the theorem in this case, it suffices to show that these points converge to $(u,v)$ as $n$ tends to infinity.

Let us write $n=ak+b$ with $0\le b <k$, and let
\(
M_n:=J T_{\gamma_{k-b+1}} \dots J T_{\gamma_k} M^a.
\)
By construction, the top row of $M_n$ is the vector $(-\overline{p_n},\overline{r_n},-\overline{q_n})$

Thus by Lemma \ref{lem:qrelation}, applying the above matrix to the point $(u,v)$, we have
\(
\norm{-\overline{p_n}+\overline{r_n} u -\overline{q_n} v} = \norm{ \frac{1}{(v v_1 v_2\dots v_{k-1})^a \cdot v_{k-1} v_{k-2} \dots v_{k-b} }  }.
\)
But then we have
\(
d\left((u,v),\left( \frac{r_n}{q_n},\frac{p_n}{q_n}\right) \right) = \norm{ \frac{ \overline{p_n}-\overline{r_n} u +\overline{q_n} v}{q_n}}^{1/2}
\)
and so this clearly tends to $0$, provided $q_n$ is non-zero for sufficiently large $n$.

Now $M_n$ cannot equal $M_m$ if $m\equiv n \pmod{k}$: if instead $M_n=M_m$ with $n =ak+b$, $m=a'k+b$, and $a'>a$, then this would imply that $I =M_m(M_n)^{-1} =(J T_{\gamma_{k-b+1}} \dots J T_{\gamma_k} ) M^{a'-a}(J T_{\gamma_{k-b+1}} \dots J T_{\gamma_k} )^{-1}$, and thus that $M^{a'-a}=I$, which is a contradiction.

So suppose $q_n=0$ with $n=ak+b$ as before, then $M_n=D T_{\Gamma}$ for some diagonal matrix $D$ and some $\Gamma\in \Sieg(\Z)$. But this then implies that $h= T_{\Gamma}^{-1} D^{-1} J T_{\gamma_{k-b+1}} \dots J T_{\gamma_k}h$. Therefore $\Gamma$ is uniquely determined by $D$ and $b$, so by the previous paragraph, we see that $q_n=0$ only finitely many times.

Case 2: $|vv_1 \dots v_{k-1}| < 1$. Then by applying Lemma \ref{lem:qrelation}, we have that
\(
d\left( (u,v), \left( \frac{{R}^{(i)}}{{Q}^{(i)}},\frac{{P}^{(i)}}{{Q}^{(i)}}\right)\right) &=\norm{\frac{v\overline{Q^{(i)}} -u \overline{R^{(i)}} + \overline{P^{(i)}} } {\overline{Q^{(i)}}}}^{1/2}\\
&= \norm{ \frac{1}{\overline{Q^{(i)}} ( {Q'}^{(i)}+\frak{Q}^{(i)}u-Q^{(i)}v) }}^{1/2}
\\ &=\norm{\frac{(v v_1 v_2 \dots v_{k-1})^i}{{Q^{(i)}}} }^{1/2} \rightarrow 0.
\)
By Lemma \ref{lem:inverse}, we may replace $M^i$ with$M^{-i}=J(M^i)^\dagger J$. The vector across the top row of $M^{-i}$ is given by $(-\overline{P^{(i)}},\overline{R^{(i)}},-\overline{Q^{(i)}})$. The remainder of the proof is identical to the previous case.

\end{document}